\documentclass[reqno]{amsart}
\usepackage[utf8]{inputenc}
\usepackage[letterpaper, portrait, margin=1.2in]{geometry}
\usepackage[english]{babel}
\usepackage{amsthm}
\usepackage[utf8]{inputenc}
\usepackage{graphicx}
\usepackage{mathtools}
\usepackage{amsmath}
\usepackage{amssymb}
\usepackage{tabto}
\usepackage{enumitem}
\newtheorem*{theorem*}{Theorem}
\newtheorem{theorem}{Theorem}[section]
\newtheorem{corollary}[theorem]{Corollary}
\newtheorem{lemma}[theorem]{Lemma}
\newtheorem{remark}[theorem]{Remark}
\newtheorem{conjecture}[theorem]{Conjecture}
\theoremstyle{definition}
\newtheorem{definition}[theorem]{Definition}
\newtheorem{example}[theorem]{Example}
\DeclareMathOperator{\des}{\mathrm{des}}
\DeclareMathOperator{\maxwt}{\mathrm{maxwt}}
\usepackage[skip=0.5\baselineskip]{caption}

\title{On Permutation Weights and $\MakeLowercase{q}$-Eulerian Polynomials}
\author{Aman Agrawal} \address{Department of Physics, Indian Institute of Science, Bengaluru, Karnataka 560012, India} \email{amanagrawal@iisc.ac.in}
\author{Caroline Choi} \address{Department of Mathematics, Stanford University, Stanford, CA 94305} \email{cchoi1@stanford.edu}
\author{Nathan Sun} \address{Department of Mathematics, Harvard University, Cambridge, MA 02138} \email{nsun@college.harvard.edu}
\date{June 6, 2020}

\begin{document}
\maketitle

\begin{abstract}
Weights of permutations were originally introduced by
Dugan, Glennon, Gunnells, and Steingr\'\i msson \cite{dugan2019tiered} in their study of the combinatorics of tiered trees. Given a permutation $\sigma$ viewed as a sequence of integers, computing the weight of $\sigma$ involves recursively counting descents of certain subpermutations of $\sigma$. Using this weight function, one can define a $q$-analog $E_n(x,q)$ of the Eulerian polynomials. We prove two main results regarding weights of permutations and the polynomials $E_n(x,q)$. First, we show that the coefficients of $E_n(x, q)$ stabilize as $n$ goes to infinity, which was conjectured by Dugan, Glennon, Gunnells, and Steingr\'\i msson \cite{dugan2019tiered}, and enables the definition of the formal power series $W_d(t)$, which has interesting combinatorial properties. Second, we derive a recurrence relation for $E_n(x, q)$, similar to the known recurrence for the classical Eulerian polynomials $A_n(x)$. Finally, we give a recursive formula for the numbers of certain integer partitions and, from this, conjecture a recursive formula for the stabilized coefficients mentioned above.
\end{abstract}

\section{Introduction}

Dugan, Glennon, Gunnells, and Steingr\'\i msson  defined certain weights of permutations in their work on the combinatorics of tiered trees \cite{dugan2019tiered}. Tiered trees are a generalization of maxmin trees, which were originally introduced by Postnikov \cite{postnikov1997intransitive} and appear in the study of local binary search trees \cite{forge2014linial}, hypergeometric functions \cite{gelfand1997combinatorics}, and hyperplane arrangements \cite{postnikov2000deformations, stanley1996hyperplane}. In \cite{dugan2019tiered}, Dugan, Glennon, Gunnells, and Steingr\'\i msson defined weight for tiered trees, motivated by their role in the enumerations of absolutely irreducible representations of certain quivers and torus orbits on certain homogeneous varieties \cite{gunnells2018torus}. They showed that weight $0$ maxmin trees are in natural bijection with permutations. They defined the weight of a permutation $\sigma$ as the largest weight of a maxmin tree that can be constructed from $\sigma$. Computing the weight of $\sigma$ involves recursively counting descents of certain subpermutations of $\sigma$.

Using the weight of a permutation, a new $q$-analog of the Eulerian polynomials $E_n(x,q)$ was defined in \cite{dugan2019tiered}. The $q$-Eulerian polynomial $E_n(x,q)$ differs from other $q$-Eulerian polynomials by Carlitz \cite{carlitz1954}, Foata-Sch{\"u}tzenberger \cite{foata1978major}, Stanley \cite{stanley1976binomial}, and Shareshian-Wachs \cite{shareshian2007}, and presents interesting combinatorial properties. It was observed in \cite{dugan2019tiered} that the coefficients of $E_n(x,q)$ seemed to exhibit a certain $\textit{stabilization phenomenon}$. From this conjecture, the formal power series $W_d(t)$ was extracted from $E_n(x,q)$ and observed to have a connection with enumerations of a certain type of partition $T(n,k),$ see sequence A256193 in \cite{sloane2003line}.

We present two main results regarding permutation weights and the $q$-Eulerian polynomials $E_n(x,q)$. We prove the stabilization phenomenon conjectured in \cite{dugan2019tiered}, which gives an explicit formula for the formal power series $W_d(t)$ defined in \cite{dugan2019tiered}. We derive a recurrence relation for the $q$-Eulerian polynomials $E_n(x,q)$, similar to the known recurrence for the classical Eulerian polynomials $A_n(x)$.

This paper is organized as follows. In Section \ref{sec:preliminaries}, we introduce preliminary definitions and notation. In Section \ref{sec:stabilization}, we define the weight disparity of a permutation and find a lower bound for weight disparity. We then use this result to show that the coefficients of $x^d$ in the formal power series $W_d(t)$ do indeed stabilize. In Section \ref{sec: formula}, we derive a recurrence relation for $E_n(x,q)$. In Section \ref{sec:second proof}, we determine the conditions for which coefficients of $E_n(x,q)$ are stabilized and give a second proof of the stabilization phenomenon. We conclude with a conjecture regarding the coefficients of $W_d(t)$ in Section \ref{sec: OEIS}.

\section{Preliminaries} \label{sec:preliminaries}
Let $S_n$ denote the set of permutations of $[n] = \{1,2, \dots, n\}$. We consider the permutation $\sigma \in S_n$ as an ordered sequence of integers $a_1 a_2 \dots a_n$. We say that $\sigma \in S_n$ is a permutation of length $n$. We say that $i \in [n-1]$ is a descent of $\sigma$ if $a_i > a_{i+1}$, and we write $\des(\sigma)$ to denote the number of descents in $\sigma$. The definitions in this section follow from \cite{dugan2019tiered}.

\begin{definition}\label{def:2.1}
Given a permutation $\sigma = a_{1} a_{2} \cdot\cdot\cdot a_{n}$, we define a method for splitting $\sigma$ into certain subpermutations $\sigma_{1}, \sigma_{2}, \dots, \sigma_{j}$.

\begin{enumerate}[label=(\roman*)]
\item
Find the minimum element of $\sigma$, call it $a_{m}$. Divide $\sigma$ into subpermutations $\sigma_{\ell}$, $a_{m}$, and $\sigma_{r}$, where $\sigma_{\ell} = a_{1} a_{2} \cdot \cdot \cdot a_{m-1}$ and $\sigma_{r} = a_{m+1} a_{m+2} \cdot\cdot\cdot a_n$. We have $\sigma = \sigma_{\ell}$ $\cdot$ $a_{m}$ $\cdot$ $\sigma_{r}$.

\item
Find the largest element of $\sigma_{\ell}$, call it $a_{k}$. Let $\sigma_{\ell_{1}} = a_{1} a_{2} \cdot\cdot\cdot a_{k}$ and let $\sigma_{\ell_{2}} = a_{k+1} a_{k+2} \cdot\cdot\cdot a_{m-1}$. We have $\sigma = \sigma_{\ell_{1}}$ $\cdot$ $\sigma_{\ell_{2}}$ $\cdot$ $a_{m}$ $\cdot$ $\sigma_{r}$.

\item
Repeat step (ii) for $\sigma_{\ell_{2}}$ until $\sigma_{\ell}$ cannot be divided further. This process results in a collection of subpermutations $\sigma_{\ell_{1}}, \sigma_{\ell_{2}}, \dots, \sigma_{\ell _{s}}, \sigma_{r}$.
\end{enumerate}
\end{definition}

\begin{example} \label{ex:2.2}
We split the permutation $\sigma = 839562147$.

By step (i), we have $839562 \cdot 1 \cdot 47$.

By step (ii), we have $839 \cdot 562 \cdot 1 \cdot 47$.

By step (iii), we have $839 \cdot 56 \cdot 2 \cdot 1 \cdot 47$.

Since $\sigma$ cannot be split further, we are done. 
\end{example}

\begin{definition}\label{def:2.3}
Let $\sigma \in S_{n}$. Then the $\textit{weight}$ of $\sigma$ is the function $w: S_{n} \rightarrow \mathbb{Z}^{\geq 0}$ defined recursively as follows:

\begin{enumerate}[label=(\roman*)]
\item
If $\sigma$ is the identity permutation or $n=1$, then $w(\sigma)=0$. 

\item
Otherwise, consider $\sigma$ as a permutation of $S_{n+1}$ by appending $n+1$ to the right of $\sigma$. Let $\sigma_{1}, \sigma_{2}, \dots, \sigma_{j}$ be the subpermutations of this new permutation as in Definition \ref{def:2.1}. Then the weight of $\sigma$ is defined as $$w(\sigma) = \sum_{i=1}^{j} (\des(\sigma_{i}) + w(\sigma_{i})).$$
\end{enumerate}
\end{definition}

\begin{example}\label{ex:2.4}
Consider the permutation $\sigma = 781659243 \in S_9$. We complete $\sigma$ to $781659243A$ by adding a maximal element A to the end. After splitting, we have $78 \cdot 1 \cdot 659243A$. The weight of $\sigma$ is given by $$w(\sigma) = (w(78) + \des(78)) + (w(1) + \des(1)) + (w(659243A) + \des(659243A)).$$
We compute $w(\sigma)$ as follows:
\begin{itemize}
    \item Since the permutation $1$ is of length 1, we have $w(1) = 0$ and $\des(1)=0$.
    \item It is clear that that $w(78)=0$ and $\des(78)=0$.
    \item If we consider the numbers in the permutation $659243A$ relative to each other, we can see that $659243A$ flattens to the permutation $5461327$. We complete the permutation to $54613278$ and split to get $546 \cdot 1 \cdot 3278$. So the weight of $659243A$ is given by $w(659243A) = (w(546) + \des(546)) + (w(1) + \des(1)) + (w(3278)+\des(3278))$. A quick computation shows that $w(546)=0$ and $w(3278)=0$, so we have $w(65392) = (0+1)+(0+0)+(0+1)=2$.
\end{itemize}
Thus $w(\sigma) = (0+0) + (0+0) + (2+3) = 5$.
\end{example}

\begin{definition}\label{def:2.5}
The $\textit{Eulerian polynomial}$ $A_{n}(x)$ is defined as $$A_{n}(x) = \sum_{\sigma \in S_{n}} x^{\des(\sigma)}.$$ The coefficient of $x^m$ in $A_{n}(x)$ count the permutations in $S_n$ with $m$ descents and are the $\textit{Eulerian numbers}$, denoted $A(n,m)$. 
\end{definition}

Using the permutation statistic of weight, a new generalization of the Eulerian polynomials $E_n(x,q)$ was defined in \cite{dugan2019tiered}.

\begin{definition}\label{def:2.6}
The $\textit{\textit{q}-Eulerian polynomial}$ $E_n(x,q)$ is defined as $$E_{n}(x,q) = \sum_{\sigma \in S_{n}} x^{\des(\sigma)}q^{w(\sigma)}.$$
\end{definition}

\begin{example}\label{ex:2.7}
We list the $q$-Eulerian polynomials $E_n(x,q)$ up to $n=7$: 

$E_{0}(x,q) = 1$ \\

$E_{1}(x,q) = 1$ \\

$E_{2}(x,q) = 1+x$ \\

$E_{3}(x,q) = 1+x(q+3) + x^2$ \\

$E_{4}(x,q) = 1+x(q^2+3q+7) + x^2(q^2 + 4q+6) + x^3$ \\

$E_{5}(x,q) = 1+x(q^3+3q^2+7q+15) + x^2(q^4 + 4q^3+11q^2 + 25q+25) + x^3(q^3 + 5q^2 + 10q + 10) + x^4$ \\

$E_{6}(x,q) = 1+x(q^4+3q^3+7q^2+15q+31) + x^2(q^6+4q^5+11q^4+31q^3+58q^2 + 107q+90) + x^3(q^6+5q^5+16q^4+34q^3 + 76q^2 + 105q + 65) + x^4(q^4+6q^3+15q^2+20q+15)+x^5$ \\

$E_{7}(x,q) = 1+x(q^5 +3q^4+7q^3+15q^2+31q+63) + x^2(q^8+4q^7+11q^6+31q^5+65q^4+149q^3+237q^2 + 392q+301) + x^3(q^9 + 5q^8+16q^7+41q^6+104q^5+203q^4+380q^3 + 609q^2 + 707q + 350) + x^4(q^8+6q^7+22q^6+55q^5+106q^4+210q^3+336q^2+315q+140)+x^5(q^5+7q^4+21q^3+35q^2+35q+21) + x^6$
\end{example}

\section{Proof of Stabilization Phenomenon}
\label{sec:stabilization}

Let us fix $k$ and consider the coefficients of $x^k$ in $E_{n}(x,q)$. It was conjectured in \cite{dugan2019tiered} that as $n$ goes to infinity, these coefficients converge to a fixed sequence and thus display a $\textit{stabilization phenomenon}$. Observe that for $k=2$ in Example \ref{ex:2.7}, the coefficients of $x^2$ seem to stabilize to
\begin{center}
$q^N+4q^{N-1}+11q^{N-2}+31q^{N-3}+65q^{N-4}+ \cdots$
\end{center}
where $N$ is the maximum weight of a permutation of length $n$ with $2$ descents. Using this conjectural observation, the formal power series $W_{d}(t) \in \mathbb{Z}[[t]]$ was extracted in \cite{dugan2019tiered} and observed to have interesting combinatorial properties.

In this section, we prove the stabilization phenomenon conjectured in \cite{dugan2019tiered}. We first define the weight disparity of a permutation and find a lower bound for weight disparity in Theorem \ref{thm:disparity}. We then use this bound to prove Theorem \ref{thm:stabilization}. We thus give an explicit formula for the formal power series $W_{d}(t)$, enabling the study of certain combinatorial properties in the coefficients of $W_d(t)$.

\begin{definition}\label{def:3.1}
We denote the \textit{maximum weight} of a permutation of length $n$ with $d$ descents by $\maxwt(n,d)$.
\end{definition}

It was shown in \cite{dugan2019tiered} that the maximum weight of such a permutation is $d(n-d-1)$.

The following lemma follows directly from the proof of Theorem 6.10 in \cite{dugan2019tiered}. A permutation that ends in $1$ is of maximum weight when its subpermutation $\pi_L$ is a single component, in which case $w(\sigma)= \des(\pi_L)+w(\pi_L) \leq (d-1)(n-d-1).$

\begin{lemma}\label{lem:3.2}
Let $\sigma \in S_n$ be a permutation with $d$ descents. If we fix the last number in $\sigma$ to be $1$, then $\maxwt(\sigma) = (d-1)(n-d-1)$.
\end{lemma}

\begin{definition}\label{def:3.3}
Let $\sigma \in S_n$ be a permutation with $d$ descents. We define the $\textit{weight disparity}$ $\Delta(\sigma)$ of $\sigma$ as $\Delta(\sigma) = \maxwt(n,d)-w(\sigma)$.
\end{definition}

We now prove a lower bound for weight disparity, $\Delta(\sigma)$.

\begin{theorem}\label{thm:disparity}
Let $\sigma \in S_n$ be a permutation with $d$ descents. If $\sigma$ does not start with $1$, then $\Delta(\sigma) \geq n-d-1$.
\end{theorem}

\begin{proof}
Let $\sigma \in S_n$ be a permutation with $d$ descents such that $\sigma = \pi_L \cdot 1 \cdot \pi_R$, where the subpermutation $\pi_L$ is nonempty. We have two cases:

\begin{enumerate}[label=(\roman*)]
\item
The subpermutation $\pi_{R}$ is empty. Then $\sigma = \pi_{L} \cdot 1$, and by Lemma 3.1, we have $w(\sigma) \leq (d-1)(n-d-1)$ so $\Delta(\sigma) \ge n-d-1$.

\item
The subpermutation $\pi_{R}$ is nonempty. Then $\sigma = \pi_{L} \cdot 1 \cdot \pi_{R}$. Let $\pi_L$ be a permutation on $\ell$ numbers with $des(\pi_L) = q$. We have $\pi_{R} \in S_{n-\ell-1}$ and $des(\pi_{R}) = d-q-1$. We proceed to bound $w(\sigma)$. Computing the weight of $\sigma$, we have $$w(\sigma) = w(\pi_{L} \cdot 1) + w(\pi_{R}) + \des(\pi_{R}),$$ which implies $$w(\sigma) \leq q(\ell-q-1) + (d-q-1)(n-\ell-d+q).$$
Since $\des(\pi_R) = d-q-1 \geq 0$ and $\des(\pi_L) = q < \ell$, we have $(q-d+1)(\ell-q-1) \leq 0$. Similarly, we have $des(\pi_R)=d-q-1<n-\ell-1$ and so $q(-n+\ell+d-q) \leq 0$. Combining these two inequalities, we get
\begin{equation}\label{eq: 1}
    (q-d+1)(\ell-q-1)+q(-n+\ell+d-q) \leq 0.
\end{equation}
Adding $(d-1)(n-d-1)$ to \eqref{eq: 1} and substituting for $w(\sigma)$, we obtain $$w(\sigma) \leq (d-1)(n-d-1),$$
or equivalently, $$\Delta(\sigma) \geq n-d-1. \eqno\qedhere$$
\end{enumerate}
\end{proof}

\begin{definition} \label{def:3.5}
    We denote the coefficient of $x^d$ in $E_n(x,q)$ by $E_n[x^d]$ and the coefficient of $x^dq^m$ in $E_n(x,q)$ by $E_n[x^dq^m]$.
\end{definition}

We now show that the coefficients of $E_{n}[x^d]$ display a stabilization phenomenon.
\begin{theorem}
\label{thm:stabilization}
    For all $k, d, m \in \mathbb{N}$ such that $m = d+k+1$ and $n \ge m$, the value of $E_n[x^dq^{\maxwt(n,d)-k}]$ is independent of $n$.
\end{theorem}

\begin{proof}
We show that for sufficiently large $n$, the number of permutations of length $n$ with a given weight disparity is constant. Let $\sigma \in S_{m+1}$ be such that $m=d+k+1$ for $k, m, d \in \mathbb{N}$. We determine the form of the permutations $\sigma$ counted by $E_{m+1}[x^dq^{\maxwt(m+1,d)-k}]$. We have the following cases:

\begin{enumerate}[label=(\roman*)]
\item
Suppose $\sigma$ is of the form  $1 \cdot \pi$, where $\pi \in S_m$ and $\des(\pi)=d$ and $w(\pi) = \maxwt(m,d)-k= (d-1)(m-d-1)$. Then we have $$w(\sigma) = w(\pi)+d = \maxwt(m+1,d)-k.$$
So $\sigma$ is counted by $E_{m+1}[x^dq^{\maxwt(m+1,d)-k}]$. Note that removing $1$ from the beginning of $\sigma$ and and subtracting 1 from every element is a bijection that gives us permutations $\pi \in S_m$ with $d$ descents and of weight equal to $w(\pi) = \maxwt(m,d)-k$. So $\sigma$ is also counted by $E_{m}[x^dq^{\maxwt(m,d)-k}]$.

\item
Suppose $\sigma $ is of the form $1 \cdot \pi$, where $\des(\pi) \neq d$ or $w(\pi) \neq d(m-d-1)-k$. It is easy to see that $\sigma$ is not counted by $E_{m+1}[x^dq^{\maxwt(m+1,d)-k}]$. 

\item
Suppose $\sigma$ does not start with $1$. Let $p = \des(\sigma)$. By Theorem \ref{thm:disparity}, we have $\Delta(\sigma) \geq m - \des(\sigma)$, which implies the following:
\begin{align*}
    \maxwt(m+1,p)-w(\sigma)+p &\geq m \\
    \maxwt(m+1,p)-w(\sigma)+p &> m-1 = d + k \\
    p-w(\sigma) &> d - (\maxwt(m+1, p)-k).
\end{align*}
So either $p \neq d$ or  $w(\sigma) \neq \maxwt(m+1, p)-k$. Thus $\sigma$ is not counted by $E_{m+1}[x^dq^{\maxwt(m+1,d)-k}]$.
\end{enumerate}

Since the number of permutations with weight disparity $k$ and $d$ descents in $S_{m+1}$ is the same as those in $S_{m}$, we have
$$E_{m}[x^dq^{\maxwt(m,d)-k}] = E_{m+1}[x^dq^{\maxwt(m+1,d)-k}].$$
It is easy to see that, by induction, the above argument shows that
$$E_{m+i}[x^dq^{\maxwt(m+i,d)-k}] = E_{m+i+1}[x^dq^{\maxwt(m+i+1,d)-k}]$$ for any $i \geq 0$.
\end{proof}

Theorem \ref{thm:stabilization} shows that the formal power series $W_d(t)$ defined in \cite{dugan2019tiered} exists and is given by $W_d(t^k) = E_{d+k+1}[x^dq^{(d-1)k}]$. We can now extract $W_d(t)$ from the stabilized coefficients of $E_n(x,q)$.

\begin{definition} \label{def:3.7}
If $E_{n}[x^d]$ stabilizes to $q^N + a_{1}q^{N-1} + a_{2}q^{N-2}+ a_{3}q^{N-3} + \cdots$, where $N=d(n-d-1)$, then the formal power series $W_d(t)$ is defined in \cite{dugan2019tiered} as 
    $$W_{d}(t) = 1 + a_{1}t + a_{2}t^2 + a_{3}t^3 + \cdots$$
\end{definition}

\begin{example} \label{ex:3.8}
By Theorem \ref{thm:stabilization} and our data for the $q$-Eulerian polynomials, we have
\begin{align*}
    W_1(t) &= 1+3t+7t^2+15t^3+31t^4+63t^5+\cdots \\
    W_2(t) &= 1+4t+11t^2+31t^3+65t^4+157t^5+\cdots\\
    W_3(t) &= 1+5t+16t^2+41t^3+112t^4+244t^5+\cdots\\
    W_4(t) &= 1+6t+22t^2+63t^3+155t^4+393t^5+\cdots \\
    W_5(t) &= 1+7t+29t^2+92t^3+247t^4+590t^5+\cdots
\end{align*}
\end{example}

\section{A Recurrence Relation for $E_n(x,q)$}
\label{sec: formula}

The classical Eulerian polynomials $A_n(x)$ in Definition \ref{def:2.5} enumerate permutations according to their number of descents. The Eulerian polynomials $A_n(x)$ satisfy the following recurrence for $n \geq 1$: 
\begin{align*}
A_0(x)&=1, \\
A_n(x) &= \sum_{i=1}^{n-1} \binom{n-1}{i} A_i(x) \cdot A_{n-i-1}(x) \cdot x.
\end{align*}

We now derive a recurrence relation for the $q$-Eulerian polynomials $E_n(x,q)$, similar to the recurrence for $A_n(x)$.

\begin{theorem} \label{thm:formula}
The $q$-Eulerian polynomials $E_n(x,q)$ satisfy the recurrence relation
    \begin{align*}
        E_{0}(qx,q) &= 1, \\
        E_n(x,q) &= \sum_{i=1}^{n-1} \binom{n-1}{i} E_i(x,q) \cdot E_{n-i-1}(qx,q) \cdot x + E_{n-1}(qx,q).
    \end{align*}
\end{theorem}

We introduce the following definitions and lemmas used in the proof of Theorem \ref{thm:formula}. 

\begin{definition} \label{def:4.2}
We define $S_n^{\prime}$ to be the set of all permutations in $S_n$ ending in $1$. That is, $S^\prime_{n} = \{\sigma \in S_n \mid \sigma = \pi \cdot 1,$ where $\pi \in S_{n-1}\}$.
\end{definition}

\begin{lemma} \label{lem:4.3}
There exists a bijection $f:S_n \leftrightarrow S^\prime_{n+1}$ such that, for all $\sigma \in S_n$, $w(f(\sigma)) = w(\sigma)$ and $\des(f(\sigma)) = \des(\sigma) + 1$.
\end{lemma}

\begin{proof}
We first show that there exists a function $f: S_n \rightarrow S^\prime_{n+1}$ that preserves weight and increases number of descents by $1$. Let $\pi \in S_n$ be such that $\pi = \pi_L \cdot 1 \cdot \pi_R$, where the subpermutations $\pi_{L}$ and $\pi_R$ can be empty. Consider the function $f: S_n \rightarrow S^\prime_{n+1}$ defined by
    $$f(\pi) = f(\pi_L \cdot 1 \cdot \pi_R) = \pi_R \cdot (n+1) \cdot \pi_L \cdot 1.$$
That is, $f$ replaces $1$ with the maximum element $(n+1)$, exchanges $\pi_L$ and $\pi_R$, and appends a minimum element $1$ at the end. Notice that for any $\pi \in S_n$, we have $f(\pi) \in S^\prime_{n+1}$ and $\des(f(\pi)) = \des(\pi) + 1$. The weight of $f(\pi)$ is given by
    $$w(f(\pi)) = w(\pi_L \cdot 1) + w(\pi_R) + \des(\pi_R),$$
so $w(f(\pi)) = w(\pi)$. So $f$ preserves weight and increases the number of descents by $1$.

Now $f$ is clearly invertible as $(n+1)$ and $1$ in $f(\pi)$ uniquely determine $\pi_L$ and $\pi_R$. Moreover, $f$ is a bijection as $S_n$ and $S^\prime_{n+1}$ are equinumerous, where its inverse $f^{-1}$ must preserve weight and reduce the number of descents by 1.
\end{proof}

\begin{definition} \label{def:4.5}
We write $E_n^{*}(x,q)$ to denote the $q$-Eulerian polynomial
        $$\sum_{\sigma \in S_n^{\prime}} x^{\des(\sigma)}q^{w(\sigma)}.$$
\end{definition}

The following corollary is a direct result of Lemma \ref{lem:4.3}.
\begin{corollary} \label{lem:4.5}
We have
\begin{equation}
    xE_n(x,q) = E_{n+1}^{*}(x,q).
\end{equation}
\end{corollary}

\begin{lemma} \label{lem:4.6}
The $q$-Eulerian polynomials $E_n(x,q)$ satisfy the recurrence relation
$$E_n[x^d] = \sum_{i=1}^{n-1}\binom{n-1}{i} \Big ( \sum_{k=1}^{i} E_{n-i-1}[x^{d-k}] \cdot E_i[x^{k-1}] \cdot q^{d-k} \Big ) + E_{n-1}[x^d]q^d.$$
\end{lemma}

\begin{proof}
Let $\sigma \in S_n$ be a permutation with $d$ descents. We have two cases:

\emph{Case 1.} Suppose $\sigma = 1 \cdot \pi_1$ for some $\pi_1 \in S_{n-1}$. We have $w(\sigma) = w(\pi_1) + d$. So the number of permutations $\sigma \in S_n$ with $d$ descents and of weight $w(\sigma)$ is equal to the number of permutations $\pi_1 \in S_{n-1}$ with $d$ descents and of weight $w(\pi_1)+d$. This is counted by $E_{n-1}[x^d]q^d$.

\emph{Case 2.} Suppose $\sigma = \pi_L \cdot 1 \cdot \pi_R$, where $1$ is in position $i+1$ of $\sigma$ and $1 \leq i \leq n-1$. It is clear that $\pi_L \in S_i$ and $\pi_R \in S_{n-i-1}$. There are $\binom{n-1}{i}$ ways to select the $i$ elements in the subpermutation $\pi_L$. Let $k = \des(\pi_L \cdot 1)$, where $1\leq k \leq i$. Then we have $\des(\pi_R)=d-k$ and $w(\sigma) = w(\pi_L \cdot 1) + w(\pi_R) + \des(\pi_R)$. The permutations $\pi_L \cdot 1$ are counted by $E^*_{i+1}[x^{k}]$, and the permutations $\pi_R$ are counted by $E_{n-1-i}[x^{d-k}]$. We know that $E_n[x^dq^w] \cdot E_{m}[x^eq^v]$ counts permutations of length $n+m$ with $d+e$ descents and of weight $w+v$. Thus permutations of length $n$ with $d$ descents and of weight $w(\pi_L \cdot 1) + w(\pi_R)$ are counted by $E^*_{i+1}[x^{k}] \cdot E_{n-1-i}[x^{d-k}].$ We add $\des(\pi_R) = d-k$ to the weight of these permutations to count permutations of weight $w(\sigma)$, length $n$, and $d$ descents. Multiplying by $q^{d-k}$, we have $q^{d-k} \cdot E^*_{i+1}[x^{k}] \cdot E_{n-1-i}[x^{d-k}].$
Thus the permutations $\sigma$ are counted by
   $$\sum_{i=1}^{n-1} \binom{n-1}{i} \sum_{k=1}^{i} E_{i+1}^{*}[x^{k}] \cdot E_{n-1-i}[x^{d-k}] q^{d-k},$$
which by Corollary \ref{lem:4.5} is equal to
   $$\sum_{i=1}^{n-1} \binom{n-1}{i} \sum_{k=1}^{i} E_{i}[x^{k-1}] \cdot E_{n-1-i}[x^{d-k}]q^{d-k}.$$
Combining Case 1 and Case 2, we have
    $$E_n[x^d]=\sum_{i=1}^{n-1} \binom{n-1}{i} \Big ( \sum_{k=1}^{i} E_{i}[x^{k-1}] \cdot E_{n-1-i}[x^{d-k}]q^{d-k} \Big ) +E_{n-1}[x^d]q^d. \eqno\qedhere$$
\end{proof}

\begin{proof}[Proof of Theorem \ref{thm:formula}]
From Lemma \ref{lem:4.6}, we have $$E_n(x,q) = \sum_{d=0}^{n-1} x^d \Big ( \sum_{i=1}^{n-1}\binom{n-1}{i} \Big ( \sum_{k=1}^{i} E_{n-i-1}[x^{d-k}] \cdot E_i[x^{k-1}] \cdot q^{d-k} \Big )  + E_{n-1}[x^d]q^d \Big ).$$
If we distribute the first summation and let $\ell = d-k$, we get
\begin{align*}
        E_n(x,q) &= \sum_{i=1}^{n-1} \binom{n-1}{i} \Big ( \sum_{k=1}^{i} \Big ( \sum_{\ell=0}^{n-1-i} E_{n-i-1}[x^{\ell}] \cdot q^\ell \cdot  E_{i}[x^{k-1}] \cdot x^{\ell+k} \Big ) \Big )\\
        &+ \sum_{j=0}^{n-1} E_{n-1}[x^j] q^j  x^j.
\end{align*}
We can rearrange terms to get
\begin{align*}
        E_n(x,q)&=\sum_{i=1}^{n-1} \binom{n-1}{i} \Big ( \sum_{k=1}^{i} E_{i}[x^{k-1}]x^{k-1} \Big ) \Big ( \sum_{\ell=0}^{n-i-1} E_{n-i-1}[x^\ell] \cdot q^\ell x^{\ell} \Big ) \cdot x\\
        &+\sum_{j=0}^{n-1} E_{n-1}[x^j]  q^j  x^j.
\end{align*}
Substituting the corresponding $q$-Eulerian polynomials, we have
$$E_n(x,q) = \sum_{i=1}^{n-1} \binom{n-1}{i} E_i(x,q) \cdot E_{n-i-1}(qx,q) \cdot x + E_{n-1}(qx,q). \eqno\qedhere$$
\end{proof}

\begin{remark}
We compare our recurrence relation for $E_n(x,q)$ with other recurrences in the literature. Note that if we set $q=1$, Theorem \ref{thm:formula} becomes the recurrence for the classical Eulerian polynomials $A_n(x)$. Since there are type B and D Eulerian polynomials that satisfy recurrence relations analogous to the recurrence for the polynomials $A_n(x)$ \cite{hyatt2016recurrences}, there may also be other $q$-analogs like the one we study.

Finally, we note that the recurrence relation in Theorem \ref{thm:formula} resembles the recurrence derived by Postnikov for the polynomials $f_n(x)$, given by $$f_n(x) = x \Big ( \sum_{\ell=1}^{n-1} \binom{n-1}{\ell} \cdot f_\ell^{\prime}(x) f_{n-\ell-1}(x) + f_{n-1}(x) \Big ),$$ where $f_n(x)= \sum_k f_{nk}x^k$ and $f_{nk}$ denotes the number of all maxmin trees on the set $[n+1]$ with $k$ maxima \cite{postnikov1997intransitive}. Given that weight 0 maxmin trees are in bijection with permutations \cite{dugan2019tiered}, it would be interesting to explore this connection between all maxmin trees and permutations further.
\end{remark}

Using Lemma \ref{lem:4.6} and Theorem \ref{thm:formula}, we give a recursive formula for the coefficients of $E_n(x,q)$.
\begin{corollary} \label{cor:4.9}
The coefficients of the $q$-Eulerian polynomials $E_n(x,q)$ satisfy the recurrence relation
\begin{align*}
        E_n[x^dq^m] &= \sum_{i=1}^{n-1} \binom{n-1}{i} \Big ( \sum_{k=1}^{i} \Big ( \sum_{j=0}^{m} E_{n-i-1}[x^{d-k}q^{m-j}] \cdot E_{i}[x^{k-1}q^{k+j-d}] \Big ) \Big )\\
        &+ E_{n-1}[x^dq^{m-d}].
\end{align*}
\end{corollary}

\begin{proof}
Since by Lemma \ref{lem:4.6}, we have $$E_n[x^d]=\sum_{i=1}^{n-1} \binom{n-1}{i} \sum_{k=1}^{i} E_{i}[x^{k-1}] \cdot E_{n-1-i}[x^{d-k}]q^{d-k}+E_{n-1}[x^d]q^d,$$ we can piece together $m$ in $E_n[x^dq^m]$ by introducing a third summation and iterating through all possibilities.
\end{proof}

\section{Alternate Proof of the Stabilization Phenomenon}
\label{sec:second proof}
Using our recursive formula for the coefficients of $E_n(x,q)$ in Corollary \ref{cor:4.9}, we determine a lower bound for the exponent of $q$ such that the coefficients $E_n[x^dq^w]$ are stabilized in Theorem \ref{thm:5.1}. We then use this result to give a second proof of Theorem \ref{thm:stabilization}.

\begin{theorem} \label{thm:5.1}
If $m > (d-1)(n-d-1)$ for $0 < d < n-1$, then $$E_n[x^dq^m] = E_{n-1}[x^dq^{m-d}].$$
\end{theorem}

\begin{proof}
By Corollary \ref{cor:4.9}, we have
\begin{equation}\label{eq:4}
    \begin{split}
        E_n[x^d&q^{m}] = E_{n-1}[x^dq^{m-d}] + \\
        & \sum_{i=1}^{n-1} \binom{n-1}{i} \Big ( \sum_{k=1}^{i} \Big (\sum_{j=0}^{m} E_{n-i-1}[x^{d-k}q^{m-j}] \cdot E_i[x^{k-1}q^{k+j-d}] \Big ) \Big ).
    \end{split}
\end{equation}
For fixed $i$ and $k$, we show that for all $j$, we have $E_{n-i-1}[x^{d-k}q^{m-j}] \cdot E_i[x^{k-1}q^{k+j-d}] = 0$. We consider the following two cases:

\emph{Case 1.} Suppose $m=(d-1)(n-d-1)+u$ for some $u \in \mathbb{N}$. We show that $E_n[x^dq^m] = E_{n-1}[x^dq^{m-d}]$. Note that $m-d \geq 0$ since $0<d<n-1$. By Corollary \ref{cor:4.9}, we have
\begin{equation} \label{eq:4}
    \begin{split}
    &E_n[x^dq^{(d-1)(n-d-1)+u}] = E_{n-1}[x^dq^{(d-1)(n-d-1)+u-d}] + \\
    & \sum_{i=1}^{n-1} \binom{n-1}{i} \Big ( \sum_{k=1}^{i} \Big (\sum_{j=0}^{(d-1)(n-d-1)+u} E_{n-i-1}[x^{d-k}q^{(d-1)(n-d-1)+u-j}]\\
    & \qquad \qquad \qquad \qquad \qquad \qquad \qquad \qquad \qquad \qquad \cdot E_i[x^{k-1}q^{k+j-d}] \Big ) \Big ).
    \end{split}
\end{equation}
The weight of a permutation of length $n$ with $d$ descents is bounded by $\maxwt(n,d)=d(n-d-1)$. So from the coefficients $E_i[x^{k-1}q^{k+j-d}]$ and $E_{n-i-1}[x^{d-k}q^{(d-1)(n-d-1)+u-j}]$, we have the following inequalities:
\begin{equation} \label{eq: ie3} 
(d-1)(n-d-1)+1+u-j \leq (d-k)(n-i-1-d+k-1)
\end{equation}
\begin{equation} \label{eq: ie4} 
k+j-d \leq (k-1)(i-k). 
\end{equation}
Analyzing the terms and summation bounds in \eqref{eq:4} gives $1 \leq k \leq d$, $k \leq i$ and $n-i-1 \geq d-k-1$. A straightforward computation shows that these summation bounds and \eqref{eq: ie3} and \eqref{eq: ie4} contradict each other. Thus for all $j$, we have $E_{n-i-1}[x^{d-k}q^{m-j}] \cdot E_i[x^{k-1}q^{k+j-d}] = 0$ and
\begin{equation*}
    \begin{split}
        \sum_{i=1}^{n-1} \binom{n-1}{i} \Big ( \sum_{k=1}^{i} \Big (\sum_{j=0}^{(d-1)(n-d-1)+u} &E_{n-i-1}[x^{d-k}q^{(d-1)(n-d-1)+u-j}]\\
        &\cdot E_i[x^{k-1}q^{k+j-d}] \Big ) \Big )=0.
    \end{split}
\end{equation*}
So $E_n[x^dq^m] = E_{n-1}[x^dq^{m-d}]$.

\emph{Case 2.} Suppose $m = (d-1)(n-d-1)-v$ for some integer $v \geq 0$. Using the same reasoning as in Case 1, we have the following inequalities:
\begin{equation} \label{eq: ie7} 
(d-1)(n-d-1)-j-v \leq (d-k)(n-1-i-d+k-1)
\end{equation}
\begin{equation} \label{eq: ie8} 
k+j-d \leq (k-1)(i-k)
\end{equation}
A straightforward computation shows that there exist $j$ that satisfy both \eqref{eq: ie7} and \eqref{eq: ie8}. Thus we have 
\begin{equation*}
    \begin{split}
        \sum_{i=1}^{n-1} \binom{n-1}{i} \Big ( \sum_{k=1}^{i} \Big (\sum_{j=0}^{(d-1)(n-d-1)-v} &E_{n-i-1}[x^{d-k}q^{(d-1)(n-d-1)-v-j}]\\
        &\cdot E_i[x^{k-1}q^{k+j-d}] \Big ) \Big ) \neq 0.
    \end{split}
\end{equation*}
So $E_n[x^dq^m] \neq E_{n-1}[x^dq^{m-d}]$.
\end{proof}

Using Theorem \ref{thm:5.1}, we give an alternate proof of Theorem \ref{thm:stabilization}.

\begin{proof}[Proof of Theorem \ref{thm:stabilization}]
By Theorem \ref{thm:5.1}, we have $E_n[x^dq^m] = E_{n-1}[x^dq^{m-d}]$ if $m > (d-1)(n-d-1)$, where $0 < d < n-1$. Furthermore, $E_{n}[x^dq^m]$ is stabilized when $m > (d-1)(n-d-1)$ and is not stabilized when $m \leq (d-1)(n-d-1)$. It is easy to show that $E_n[x^{0}q^{0}]$ is stabilized and $E_n[x^{n-1}q^{0}]$ is not.
\end{proof}

\section{Connection Between $W_d(t)$ and Integer Partitions}
\label{sec: OEIS}

Theorem \ref{thm:formula} gives a recursive formula for the $q$-Eulerian polynomials $E_n(x,q)$, and Theorem \ref{thm:5.1} states the condition for which the coefficients $E_n[x^dq^w]$ are stabilized. We conclude with a conjecture regarding the coefficients of $W_d(t)$, which are the stabilized coefficients of $E_n(x,q)$.

It was observed in \cite{dugan2019tiered} that the coefficients of $W_d(t)$ correspond to numbers in the sequence {\fontfamily{cmtt}\selectfont
A256193} in the Online Encyclopedia of Integer Sequences \cite{sloane2003line}. The numbers $T(n,k)$ count partitions of $n$ with exactly $k$ parts of a second type, denoted by a prime $^{\prime}$. For example, we have \\

\qquad \quad $T(3,0) = 3$, corresponding to 111, 21, 3,

\qquad \quad $T(3,1) = 6$, corresponding to $1^{\prime}11$, $11^{\prime}1$, $111^{\prime}$, $2^{\prime}1$, $21^{\prime}$, $3^{\prime}$,

\qquad \quad $T(3,2) = 4$, corresponding to $1^{\prime}1^{\prime}1$, $1^{\prime}11^{\prime}$, $11^{\prime}1^{\prime}$, $2^{\prime}1^{\prime}$,

\qquad \quad $T(3,3) = 1$, corresponding to $1^{\prime}1^{\prime}1^{\prime}$.
\\

Table \ref{tab:2} is a short table of $T(n,k)$ with $k$ constant along the columns. Numbers in bold correspond to coefficients of the power series $W_d(t)$.

\begin{table}[htb]
\begin{tabular}{llllllllll}
\textbf{1} &            &             &             &              &              &              &             &             &            \\
1          & \textbf{1} &             &             &              &              &              &             &             &            \\
2          & \textbf{3} & \textbf{1}  &             &              &              &              &             &             &            \\
3          & 6          & \textbf{4}  & \textbf{1}  &              &              &              &             &             &            \\
5          & 12         & \textbf{11} & \textbf{5}  & \textbf{1}   &              &              &             &             &            \\
7          & 20         & 24          & \textbf{16} & \textbf{6}   & \textbf{1}   &              &             &             &            \\
11         & 35         & 49          & \textbf{41} & \textbf{22}  & \textbf{7}   & \textbf{1}   &             &             &            \\
15         & 54         & 89          & 91          & \textbf{63}  & \textbf{29}  & \textbf{8}   & \textbf{1}  &             &            \\
22         & 86         & 158         & 186         & \textbf{155} & \textbf{92}  & \textbf{37}  & \textbf{9}  & \textbf{1}  &            \\
30         & 128        & 262         & 351         & 342          & \textbf{247} & \textbf{129} & \textbf{46} & \textbf{10} & \textbf{1}
\end{tabular}
\vspace{+.2cm}
\caption{The numbers $T(n,k)$} \label{tab:2}
\vspace{-.4cm}
\end{table}

We will give a recursive formula for $T(d+k,d)$, but we first need the following lemma.

\begin{lemma} \label{lem:6.1}
Let $n, k, b \in \mathbb{N}$ such that $b \leq 2k-n$. The numbers $T(n,k)$ satisfy the relation $$T(n,k)=\sum_{j=0}^{b} \binom{b}{j} \cdot T(n-b, k-j).$$
\end{lemma}

\begin{proof}
Let $n, k, b \in \mathbb{N}$ such that $b \leq 2k-n$. We first show that $T(n,k) \geq \sum_{j=0}^{b} \binom{b}{j} \cdot T(n-b, k-j)$. We then show that $T(n,k) \leq \sum_{j=0}^{b} \binom{b}{j} \cdot T(n-b, k-j)$.

To each partition counted by $T(n-b,k-j)$, we append all possible partitions consisting of $b-j$ 1s and $j$ $1^{\prime}$s and obtain partitions of $n$ with $k$ parts of a second type. For fixed $j$, there are $\binom{b}{j} \cdot T(n-b, k-j)$ such partitions. Summing over all possible values of $j$, there are $\sum_{j=0}^{b} \binom{b}{j} \cdot T(n-b, k-j)$ such partitions formed in this way. Thus $T(n,k) \geq \sum_{j=0}^{b} \binom{b}{j} \cdot T(n-b, k-j)$.

We now show that this summation counts all partitions which are counted by $T(n,k)$. Since $n < 2k$, each partition represented by $T(n,k)$ contains at least one singleton part (1 or $1^{\prime}$). In particular, the partition with the fewest singleton parts is the partition $\underbrace{2^{\prime}2^{\prime} \dots 2^{\prime}}_{n-k} \underbrace{1^{\prime}1^{\prime} \dots 1^{\prime}}_{2k-n}$. So every partition counted by $T(n,k)$ ends in at least $2k-n$ singleton parts. This means that for $b < 2k-n$, truncating the last $b$ elements from a partition which has $b-j$ 1s and $j$ $1^{\prime}$s results in a partition counted by $T(n-b,k-j)$. As there can be at most $\binom{b}{j}$ such partitions, we have $T(n,k) = \sum_{j=0}^{b} \binom{b}{j} \cdot T(n-b, k-j)$.
\end{proof}

\begin{theorem} \label{thm:w_d(t)}
For fixed $k \in \mathbb{N}$ and $d \geq 2k$, we have $$T(d+k, d) = \sum_{i=1}^{k} \Big ((-1)^{i+1} \cdot \binom{k}{i} \cdot T(d+k-i, d-i)\Big) + 1$$.
\end{theorem}

\begin{proof}
Let $k \in \mathbb{N}$ be fixed and $d \geq 2k$. By Lemma \ref{lem:6.1}, we get $$T(d+k,d) = \sum_{j=1}^{k} \binom{k}{j} \cdot T(d,d-j) + 1.$$
Note that $\binom{k}{j} = \sum_{i=1}^{j} \binom{k}{j} \binom{j}{i} (-1)^{i+1}$. We have $$T(d+k,d) = \sum_{j=1}^{k} T(d,d-j) \Big ( \sum_{i=1}^{j} \binom{k}{j} \binom{j}{i} (-1)^{i+1} \Big ) + 1.$$
By the identity $\binom{r}{s} \binom{s}{t} = \binom{r}{t} \binom{r-t}{s-t}$, we get $$T(d+k,d) = \sum_{j=1}^{k} T(d,d-j) \Big (\sum_{i=1}^{j} \binom{k}{i} \binom{k-i}{j-i} (-1)^{i+1} \Big ) + 1.$$
If we change the order of summations and let $a = j-i$, we obtain $$T(d+k,d)= \sum_{i=1}^{k} \binom{k}{i} (-1)^{i+1} \Big (\sum_{a=0}^{k-i} \binom{k-i}{a} \cdot T(d,d-a-i) \Big ) + 1.$$
Since $k-i \leq 2(d-i)-(d+k-i)$, by Lemma \ref{lem:6.1} we have $$T(d+k,d) = \sum_{i=1}^{k} \Big ( (-1)^{i+1} \binom{k}{i} \cdot T(d+k-i,d-i) \Big ) + 1. \eqno\qedhere$$
\end{proof}

\begin{conjecture} \label{conj:6.1}
For fixed $k \in \mathbb{N}$ and $d \geq 2k$, we have $$W_d[t^{k}] = \sum_{i=1}^{k} \Big((-1)^{i+1} \cdot \binom{k}{i} \cdot W_{d-i}[t^{k}] \Big) + 1.$$
\end{conjecture}

\section*{Acknowledgements}
We give special thanks to Professor Einar Steingr\'\i msson for his support and invaluable advice on this paper. We are grateful to Roger Van Peski for his dedicated mentoring and copious feedback. We thank Professor Paul E. Gunnells for proposing this research topic and for his guidance. Finally, we thank the PROMYS program and the Clay Mathematics Institute, under the auspices of which we began this research.

\newpage
\bibliography{references}

\begin{thebibliography}{10}

\bibitem{carlitz1954}
L.~Carlitz.
\newblock $q$-bernoulli and eulerian numbers.
\newblock {\em Transactions of the American Mathematical Society},
  76(2):332--350, 1954.

\bibitem{dugan2019tiered}
W.~Dugan, S.~Glennon, P.~E. Gunnells, and E.~Steingr{\'\i}msson.
\newblock Tiered trees, weights, and $q$-eulerian numbers.
\newblock {\em Journal of Combinatorial Theory, Series A}, 164:24--49, 2019.

\bibitem{foata1978major}
D.~Foata and M.-P. Sch{\"u}tzenberger.
\newblock Major index and inversion number of permutations.
\newblock {\em Mathematische Nachrichten}, 83(1):143--159, 1978.

\bibitem{forge2014linial}
D.~Forge.
\newblock Linial arrangements and local binary search trees.
\newblock {\em arXiv preprint arXiv:1411.7834}, 2014.

\bibitem{gelfand1997combinatorics}
I.~M. Gelfand, M.~I. Graev, and A.~Postnikov.
\newblock Combinatorics of hypergeometric functions associated with positive
  roots.
\newblock In {\em The Arnold-Gelfand Mathematical Seminars}, pages 205--221.
  Springer, 1997.

\bibitem{gunnells2018torus}
P.~E. Gunnells, E.~Letellier, and F.~R. Villegas.
\newblock Torus orbits on homogeneous varieties and kac polynomials of quivers.
\newblock {\em Mathematische Zeitschrift}, 290(1-2):445--467, 2018.

\bibitem{hyatt2016recurrences}
M.~Hyatt.
\newblock Recurrences for eulerian polynomials of type b and type d.
\newblock {\em Annals of Combinatorics}, 20(4):869--881, 2016.

\bibitem{postnikov1997intransitive}
A.~Postnikov.
\newblock Intransitive trees.
\newblock {\em Journal of Combinatorial Theory, Series A}, 79(2):360--366,
  1997.

\bibitem{postnikov2000deformations}
A.~Postnikov and R.~P. Stanley.
\newblock Deformations of coxeter hyperplane arrangements.
\newblock {\em Journal of Combinatorial Theory, Series A}, 91(1-2):544--597,
  2000.

\bibitem{shareshian2007}
J.~Shareshian and M.~Wachs.
\newblock $q$-eulerian polynomials: excedance number and major index.
\newblock {\em Electronic Research Announcements of the American Mathematical
  Society}, 13(4):33--45, 2007.

\bibitem{sloane2003line}
N.~J. Sloane et~al.
\newblock The on-line encyclopedia of integer sequences, 2003.

\bibitem{stanley1976binomial}
R.~P. Stanley.
\newblock Binomial posets, m{\"o}bius inversion, and permutation enumeration.
\newblock {\em Journal of Combinatorial Theory, Series A}, 20(3):336--356,
  1976.

\bibitem{stanley1996hyperplane}
R.~P. Stanley.
\newblock Hyperplane arrangements, interval orders, and trees.
\newblock {\em Proceedings of the National Academy of Sciences},
  93(6):2620--2625, 1996.

\end{thebibliography}
\bibliographystyle{abbrv}

$\\ \\ \\$
\end{document}